\documentclass[11pt]{amsart}

\usepackage[english]{babel}
\usepackage{amsmath}
\usepackage{amsfonts}
\usepackage{amssymb}
\usepackage{fullpage}
\usepackage{amssymb}
\usepackage{graphicx}
\usepackage{marginnote}
\usepackage{amsopn}
\usepackage[toc,page]{appendix}
\usepackage{amsthm}
\usepackage{diffcoeff}
\usepackage{mathdots}
\usepackage{hyperref}
\usepackage{subfigure}
\usepackage[all,cmtip]{xy}
\usepackage{tikz}
\usepackage{tikz-cd}
\usetikzlibrary{babel}
\usepackage[all]{xy}
\usetikzlibrary{quotes,angles}

\DeclareMathOperator{\SO}{\textbf{SO}}

\DeclareMathOperator{\SU}{\textbf{SU}}
\DeclareMathOperator{\SL}{\textbf{SL}}

\DeclareMathOperator{\Id}{\text{Id}}

\DeclareMathOperator{\End}{\text{End}}
\DeclareMathOperator{\R}{\mathbb{R}}
\DeclareMathOperator{\Sp}{\mathbb{S}}
\DeclareMathOperator{\C}{\mathbb{C}}

\DeclareMathOperator{\CPone}{\mathbb{CP}^1}
\DeclareMathOperator{\Z}{\mathbb{Z}}

\DeclareMathOperator{\Ker}{\text{Ker}}
\DeclareMathOperator{\degree}{\text{deg}}

\DeclareMathOperator{\rank}{\text{rank}}

\DeclareMathOperator{\Res}{\text{Res}}

\DeclareMathOperator{\pardeg}{\text{par-deg}}

\DeclareMathOperator{\slb}{\mathfrak{sl}}

\newtheorem{theorem}{Theorem}[section]

\newtheorem{proposition}[theorem]{Proposition}

\theoremstyle{definition}
\newtheorem{definition}[theorem]{Definition}

\theoremstyle{remark}

\numberwithin{equation}{section}

\begin{document}

\title{DPW potentials for compact symmetric CMC surfaces in $\Sp^3$}

\begin{abstract}
Inspired by the work of Heller \cite{He1:tesi}, we show that there exists a DPW potential for the Lawson surface $\xi_{k-1, l-1}$ from which it is possible to reconstruct the minimal immersion $f:  \xi_{k-1, l-1} \to \Sp^3$ via the DPW method.
Moreover, we extend the result to surfaces immersed in the 3-sphere with constant mean curvature which satisfy a certain symmetric condition.
\end{abstract}
\maketitle

\author{\begin{center}Benedetto Manca\footnote{Dipartimento di Matematica e Informatica, Universit\`a degli Studi di Cagliari, Via Ospedale 72, 09124 Cagliari \\ Email Address: bmanca@unica.it}\end{center}}


\date{\today}


\section{Introduction and statement of the main results}
\label{SecIntro}
A natural variational problem in the geometry of surfaces is the \textit{isoperimetric problem} which consists in finding the surface of minimum area among all the surfaces enclosing a fixed volume. The answer to this problem is the round sphere and a proof of this has been known for a long time.

A more general question in the study of surfaces immersed in a three-dimensional space form is to determine the surfaces whose area is critical either under volume preserving deformations or under deformations with no constrains on the volume. The differential equation characterizing the first class of surfaces is $H=const.$, where $H$ is the mean curvature function, and are called constant mean curvature surfaces (CMC).
The only compact surface CMC embedded in $\R^3$ or $\mathbb{H}^3$ is the round sphere \cite{Al:tesi, Fer:tesi}. The situation is different if one considers CMC or minimal embeddings in $\Sp^3$. In fact it has been proved the existence of compact surfaces minimally embedded in $\Sp^3$ for every genus $g$ \cite{La:tesi, KPS:tesi}.

Given a minimal or CMC immersion $f: M \hookrightarrow \Sp^3$ of a compact Riemann surface $M$, the condition $H = const.$ is an elliptic PDE difficult to solve. The integrable systems method for minimal or CMC surfaces in $\Sp^3$ aims at interpret the condition $H = const.$ as a system of ODEs which is easier to solve.

In the case $H=0$, the \textit{associated family of flat $\SL(2, \mathbb{C}$)-connections} of the immersion $f$, introduced by Hitchin in 1990 \cite{Hit:tesi}, is given by
\begin{equation}
\label{assfamconnect}
\nabla^\lambda := \nabla + \lambda^{-1}\Phi - \lambda \Phi^*, \quad \lambda \in \mathbb{C}^*
\end{equation}
and it is defined over a hermitian rank 2 vector bundle $E \to M$. The flat connection $\nabla$ is defined as $\nabla:= d + \frac{1}{2}f^{-1}df$, the \textit{Higgs field} $\Phi \in H^0(M, \End_0(E)\otimes K)$ is the $(1,0)$-component of the 1-form $f^{-1}df$ and $\Phi^*$ is the $(0, 1)$-component of $f^{-1}df$.

For a CMC immersion $f$ with $H\neq 0$, it is still possible to define its associated family of flat $\SL(2, \mathbb{C})$-connections as follows: let $f^{-1}df = \Psi - \Psi^*$ be the decomposition of $f^{-1}df$ into $\mathbb{C}$-linear and $\mathbb{C}$-antilinear components. Define $\lambda_2 := (-iH + 1)/(iH + 1)$ and
\begin{equation*}
\Phi = \frac{\lambda_2}{1 + \lambda_2}\Psi, \quad\quad \Phi^* = \frac{1}{1 + \lambda_2}\Psi^*.
\end{equation*}
Then, the associated family of flat connections of $f$ can be written in the form \eqref{assfamconnect} with $\nabla = d + \Phi - \Phi^*$.

Both in the minimal and CMC case, the associated family of flat connections $\nabla^\lambda$ satisfies the following conditions:
\begin{itemize}
\item[$(a)$]
the connection $\nabla^\lambda$ is unitary for every $\lambda \in \Sp^1\subset \mathbb{C}^*$;
\item[$(b)$]
there exist two distinct points $\lambda_1, \lambda_2 \in \Sp^1$ such that $\nabla^{\lambda_i}$ is trivial for $i = 1, 2$;
\item[$(c)$]
the Higgs field $\Phi$ is nilpotent and nowhere vanishing.
\end{itemize}

Moreover, the following theorem due to Hitchin and Bobenko shows that it is possible to reconstruct the CMC immersion from its associated family of flat connections:
\begin{theorem}[\cite{Hit:tesi, Bo:tesi}]
Given a family of flat $\SL(2, \mathbb{C})$-connections $\nabla^\lambda$ of the form \eqref{assfamconnect} which satisfies conditions $(a) - (c)$, the gauge transformation between the trivial connections $\nabla^{\lambda_1}$ and $\nabla^{\lambda_2}$ gives a CMC immersion $f: M \hookrightarrow \Sp^3$ with mean curvature 
$$
H = i\frac{\lambda_1 + \lambda_2}{\lambda_1 - \lambda_2},
$$
whose associated family of flat $\SL(2, \mathbb{C})$-connections is the family $\nabla^\lambda$.
\end{theorem}

Hitchin \cite{Hit:tesi} classified all CMC immersions of a compact Riemann surface of genus 1 into $\Sp^3$. Unfortunately, this approach relies on the fact that the fundamental group of the surface is abelian and cannot be generalized to higher genus surfaces.

More recently, see Theorem \ref{DPWcpt} below, it has been proven that, for compact Riemann surfaces of genus $g \geq 2$ CMC immersed in $\Sp^3$, it is better to consider the associated family of flat $\SL(2, \mathbb{C})$-connections up to gauge transformations.
Given a compact Riemann surface $M$ of genus $g \geq 2$ CMC immersed in $\Sp^3$, let $\mathcal{A}^2(M)$ be the moduli space of flat connections defined on the same hermitian vector bundle $E \to M$ where the associated family of flat connections $\nabla^\lambda$ is defined.

\begin{theorem}[{\cite{HeHeS:tesi}}]
\label{DPWcpt}
Let $M$ be a compact Riemann surface of genus $g \geq 2$ and $\mathcal{D}: \mathbb{C}^* \to \mathcal{A}^2(M)$ a holomorphic map satisfying
\begin{itemize}
\item[$(1)$]
the unit circle $\Sp^1 \subset \mathbb{C}^*$ is mapped into the set consisting of gauge equivalence classes of unitary flat connections;
\item[$(2)$]
around $\lambda=0$, there exists a local lift $\tilde{\nabla}^\lambda$ of $\mathcal{D}$ with an expansion in $\lambda$
\begin{equation}
\label{loclift}
\tilde{\nabla}^\lambda \sim \lambda^{-1}\Psi + \tilde{\nabla}^0  + \text{ higher order terms in } \lambda
\end{equation}
for a nilpotent $\Psi \in \Gamma(M, K \otimes \End_0(E))$;
\item[$(3)$]
there are two distinct points $\lambda_1, \lambda_2 \in \Sp^1 \subset \mathbb{C}^*$ such that $\mathcal{D}(\lambda_j)$, $j=1, 2$, represents the trivial gauge equivalence class.
\end{itemize}
Then, there exists a (possibly branched) CMC immersion $f: M \hookrightarrow \Sp^3$ inducing the map $\mathcal{D}$ as the family of gauge equivalence classes
\begin{equation}
\label{Dmap}
\mathcal{D}(\lambda) = [\nabla^\lambda], \quad \lambda \in \mathbb{C}^*,
\end{equation}
where $\nabla^\lambda$ is the associated family of flat connections of $f$.
The branch points of $f$ are given by the zeros of $\Psi$. Conversely, every CMC immersion determines a holomorphic $\mathbb{C}^*$-curve into $\mathcal{A}^2(M)$ via \eqref{Dmap}.
\end{theorem}

Theorem \ref{DPWcpt} can be used to extend the \textit{DPW method} for higher genus compact surfaces. The DPW method \cite{DPW:tesi} was introduced by Dorfmeister, Pedit and Wu in 1998 and gives a way to construct all CMC immersions into $\Sp^3$ from a holomorphic (or meromorphic, depending on the situation) $\mathbb{C}^*$-family $\eta(z, \lambda)$ of 1-forms with values in $\slb(2, \mathbb{C})$. The family $\eta(z, \lambda)$ is called \textit{DPW potential} and has expansion series in $\lambda$ of the form
$$
\eta(z, \lambda) = \sum_{j=-1}^\infty \eta_j(z)\lambda^jdz,
$$
with $\eta_{-1}(z)$ nilpotent and upper triangular (we refer to \cite{DPW:tesi} for more details).

If the local lift $\tilde{\nabla}^\lambda$ in Theorem \ref{DPWcpt} can be written as $d + \eta(z, \lambda)$ where $\eta(z, \lambda)$ is a DPW potential, then it is possible to use the DPW method to reconstruct the immersion $f: M \hookrightarrow \Sp^3$.
In \cite{He1:tesi}, this approach has been used to construct a DPW potential for the Lawson surface $\xi_{2,1}$ of genus 2, where it has been exploited the fact that the Lawson surface of genus 2 is equipped with several symmetries.

In this article we will show that a similar approach can be used  to obtain a DPW potential for the minimal surface $\xi_{k-1, l-1}$ of genus $(k-1)(l-1)$ discovered by Lawson in \cite{La:tesi}. The minimal surface $\xi_{k-1, l-1}$ is obtained from the solution of a Plateau problem on a convex geodesic polygon in $\Sp^3$ such that the adjacent edges $\gamma_j$ and $\gamma_{j-1}$ meet at an angle of the form $\pi/(d_j + 1)$ for some $d_j \in \mathbb{N}^+$. The Plateau solution is then reflected about all the edges in order to obtain a compact minimal surfaces embedded in $\Sp^3$.
From the way it is constructed, we have that $\xi_{k-1, l-1}$ is equipped with an action of $\Z_k \times \Z_l \subset \SO(4)$ which is faithful. Moreover, the quotient $\xi_{k-1, l-1}/\Z_k \times \Z_l$ is the Riemann sphere $\CPone$ and the projection to the quotient $\pi: \xi_{k-1, l-1}\to \CPone$ is a $kl$-fold covering map, branched at four points denoted by $z_1, \dots, z_4 \in \CPone$.

We are in the right position to enunciate our main result
\begin{theorem}
\label{MainTheor}
Let $\xi_{k-1, l-1}$ be the Lawson surface of genus $(k-1)(l-1)$ with symmetry group $\Z_k\times \Z_l \subset \SO(4)$. Let $\nabla^\lambda$ be the associated family of flat $\SL(2, \mathbb{C})$-connections of the minimal immersion $f: \xi_{k-1, l-1} \hookrightarrow \Sp^3$. Then, there exists a holomorphic family of logarithmic connections
\begin{equation*}
\tilde\nabla^\lambda = \lambda^{-1}\tilde\Phi + \tilde\nabla + \text{ higher order terms in $\lambda$ }
\end{equation*}
on the four punctured sphere $\CPone$, singular at the four branch points $z_1, \dots, z_4$ of $\pi : \xi_{k-1, l-1} \to \xi_{k-1, l-1}/\Z_k\times\Z_l = \CPone$, where $\tilde\Phi$ is a nilpotent $\mathfrak{sl}(2, \mathbb{C})$-valued complex linear 1-form, which satisfies the following:
\begin{itemize}
\item[$(i)$]
there exists a flat connection $\hat\nabla$ on $\xi_{k-1, l-1}$ with $\Z_2$-monodromy representation, such that the families of connections $\nabla^\lambda$ and $\pi^*\tilde\nabla^\lambda \otimes \hat\nabla$ are gauge equivalent via a family of gauge transformations $g(\lambda)$ which extends holomorphically at $\lambda=0$;
\item[$(ii)$]
there is an open neighborhood $U$ of $\lambda=0$  such that $\tilde\nabla^\lambda$ can be represented by a $\lambda$-family of Fuchsian systems for $\lambda \in U$. More specifically, for $\lambda \in U$, we have
\begin{equation}
\label{DPWpotentialexpression}
\tilde\nabla^\lambda = d + \eta(z, \lambda) = d + \sum_{j=-1}^\infty\eta_j(z)\lambda^j,
\end{equation}
where, for every $j$, $\eta_j(z)$ is a $\mathfrak{sl}(2, \mathbb{C})$-valued 1-form with simple poles at the branch points $z_1, \dots, z_4$ and holomorphic on $\CPone\smallsetminus \{z_1, \dots, z_4\}$;
\item[$(iii)$]
the map $\lambda \mapsto \eta(z, \lambda)$ extends meromorphically to $\mathbb{C}^*$ and the connection $\tilde\nabla^\lambda = d + \eta(z, \lambda)$ has unitarizable monodromy representation for every $\lambda \in \Sp^1$ such that $\eta(z, \lambda)$ does not have a pole;
\item[$(iv)$]
the eigenvalues of the local residues of $\tilde\nabla^\lambda$ are given by the eigenvalues (of the first or second factor in $\SU(2)\times\SU(2)$) of the four generators $\gamma_1, \dots, \gamma_4$ of the finite group $\Gamma \subset \SU(2)\times\SU(2)$ which double covers $\Z_k\times\Z_l$.
\end{itemize}
In particular, the the immersion $f:\xi_{k-1, l-1} \hookrightarrow \Sp^3$ can be constructed from a meromorphic DPW potential on the four punctured sphere.
\end{theorem}

The results of this article are contained in the Phd thesis of the author.
For related work on the subject see also \cite{HeHeTr:tesi}, where the authors used the DPW potentials to estimate the area of the Lawson surfaces $\xi_{k-1, l-1}$.

The article is organized as follows: in Section \ref{SecLift} we will show that it is possible to lift the $\SO(4)$-action on the Lawson surface $\xi_{k-1, l-1}$ to a $\SU(2)\times\SU(2)$-action. This will be needed to define an action on the associated family of flat $\SL(2, \mathbb{C})$-connections of the immersion $f: \xi_{k-1, l-1} \hookrightarrow \Sp^3$ and on the holomorphic vector bundle $E\to \xi_{k-1, l-1}$ where the associated family is defined.
We will also define an auxiliary Riemann surface $\tilde{M}$ which double covers $\xi_{k-1, l-1}$ on which the $\SU(2)\times\SU(2)$-action is faithful.
In Section \ref{SecParbund} we define a vector bundle over the Riemann sphere $\CPone$ which can be equipped with a $\lambda$-family of parabolic structures and a $\lambda$-family of logarithmic connections $\tilde{\nabla}^\lambda$. The existence of a family of parabolic structures will be necessary in order to write $\tilde{\nabla}^\lambda$ as a family of Fuchsian systems which will give the DPW potential from which it will be possible to reconstruct the immersion $f: \xi_{k-1, l-1}\hookrightarrow \Sp^3$ (as shown in \cite{HeHe:tesi}).
The residue $\tilde{\Phi}$ of $\tilde{\nabla}^\lambda$, at $\lambda=0$, is studied in Section \ref{SecMain} where it is proven that $\tilde{\Phi}$ gives a parabolic Higgs field satisfying appropriate conditions which ensure that around $\lambda=0$ the holomorphic vector bundle over $\CPone$ is given by $\mathcal{O}(-2) \oplus \mathcal{O}(-2)$. 
We then prove Theorem \ref{MainTheor} which shows that the family of logarithmic connections $\tilde{\nabla}^\lambda$ gives a DPW potential for the Lawson surface $\xi_{k-1, l-1}$ from which it is possible to reconstruct the immersion $f: \xi_{k-1, l-1} \hookrightarrow \Sp^3$.
Finally, in Section \ref{SecSymsurf}, we present an extension of Theorem \ref{MainTheor} to surfaces which satisfy a symmetric condition (cf. Definition \ref{symCMCsurf}).

The following diagram could be useful for the reader to summarize the mathematical objects introduced so far, together with the corresponding maps.
\begin{center}
\begin{tikzcd}
&(\tilde{E} = \tau^*E, \tau^*\nabla^\lambda)  \arrow[rr, bend left, "\tilde\pi_*"] \arrow[d]     &(E, \nabla^\lambda)\arrow[l, dashed, "\tau^*"]  \arrow[d]               &((\tilde\pi_*\tilde{E})^{\Gamma,\lambda}, \tilde\nabla^\lambda)\arrow[l, dashed, "\pi^*"]\arrow[d]\\
&\tilde{\xi}_{k-1, l-1}\arrow[r, "\tau"]  \arrow[rr, bend right, "\tilde\pi"]               &\xi_{k-1, l-1} \arrow[r,"\pi"]       			&\CPone
\end{tikzcd}
\end{center}

\textbf{Acknowledgement}. The author would like to thank S. Heller for suggesting the problem tackled in this paper and for guiding him during his stay at the University of Tübingen and the University of Hannover. The author would also like to thank  L. Heller , A. Loi, S. Montaldo, F. Pedit and N. Schmitt for the helpful discussion and suggestions. 

The author was supported by Prin 2015 – Real and Complex Manifolds; Geometry, Topology and Harmonic Analysis – Italy and by KASBA- Funded by Regione Autonoma della Sardegna.

\section{Lifting the $\SO(4)$-action to a $\SU(2)\times\SU(2)$-action}
\label{SecLift}
In the rest of the article we will denote with $M$ the Lawson surface of genus $(k-1)(l-1)$, $k, l \geq 3$ with $\SO(4)$-symmetry group $G = \Z_k\times \Z_l$.  
We briefly recall some details of the construction of $M$ that we will need in the rest of the article (for more details we refer to \cite{La:tesi}).

Let $P_1, \dots, P_{2l+2} \in M$ be the points with stabilizer group $\Z_k$ and $Q_1, \dots, Q_{2k+2}\in M$ the points with stabilizer group $\Z_l$.
The generator $g_1$ for the stabilizer group of one point $P_{2r+1}$, $r = 0, \dots, l$ acts locally as a rotation around them. The generator $g_2$ for the stabilizer group of the points $P_{2r+2}$, $r = 0, \dots, l$ acts locally as $g_1^{-1}$. The same holds for the generators $g_3$ and $g_4$ for the stabilizer groups of the points $Q_{2s+1}$ and $Q_{2s+2}$, $s = 0, \dots, k$, respectively. 

We consider the presentation of $G = \Z_k\times \Z_l$ given by
\begin{equation*}
G = \langle g_1, \dots, g_4 \mid g_1^k = g_2^k = g_3^l = g_4^l =1, g_2 = g_1^{-1}, g_4 = g_3^{-1}, g_1 \cdots g_4  =1\rangle.
\end{equation*}

If $z_1, \dots, z_4 \in \CPone$ are the branch points of the holomorphic map $\pi:M \to \CPone$ of degree $kl$, we have that
\begin{equation*}
\begin{split}
\pi^{-1}(z_1) &= \{P_1, P_3, \dots, P_{2l+1}\}\\
\pi^{-1}(z_2) &= \{P_2, P_4, \dots, P_{2l+2}\}\\
\pi^{-1}(z_3) &= \{Q_1, Q_3, \dots, Q_{2k+1}\}\\
\pi^{-1}(z_4) &= \{Q_2, Q_4, \dots, Q_{2k+2}\}.
\end{split}
\end{equation*}

In order to construct a DPW potential for the surface $M$ with $\SO(4)$-symmetry group $G$, we will need an action of a subgroup $\Gamma \subset \SU(2) \times \SU(2)$ which acts on $M$ as the group $G$.

The subgroup $\Gamma$ is obtained as the pre-image of $G$ under the double spin covering $\varepsilon: \SU(2) \times \SU(2) \to \SO(4)$. Since, in our situation, $G = \Z_k\times \Z_l$ we have $\Gamma = \Z_{2k}\times\Z_{2l} \subset \SU(2)\times \SU(2)$.

We will need a faithful group action in order to construct the DPW potential for $M$. Since $\Gamma$ does not act faithfully on $M$ (for example $-\Id \in \Gamma$ acts on $M$ as the identity element, fixing all the points on the surface), we show that there exists a Riemann surface $\tilde{M}$, double covering $M$, on which $\Gamma$ acts faithfully.

Consider $z_0 \in \CPone \smallsetminus \{z_1, \dots, z_4\}$, where $z_1, \dots, z_4$ are the branch points of $\pi$, and let
\begin{equation}
\label{Mmonodromy}
\rho: \pi_1(\CPone \smallsetminus \{z_1, \dots, z_4\}, z_0)\to G
\end{equation}
be its monodromy representation. A simple loop $\eta_j$ around the point $z_j$ is mapped to the generator $g_j$ of $G$, for $j = 1, \dots, 4$.

Geometrically, we can interprete this as follows: If $d_j$ is the order of the element $g_j \in G$, the loop $\eta_j^{d_j}$ in $\CPone$ is lifted to a closed loop around the points in $\pi^{-1}(z_j)$. Thus, the surface $M$ closes, locally around each of the point in $\pi^{-1}(z_j)$, after $d_j$ rotations of an angle $\frac{2\pi}{d_j}$.

Let $\gamma_j \in \varepsilon^{-1}(g_j)$, $j=1, \dots, 4$, we can define a group homomorphism 
\begin{equation}
\label{tildemonodromy}
\tilde{\rho}: \pi_1(\CPone \smallsetminus \{z_1, \dots, z_4\}, z_0)\to \Gamma
\end{equation}
which maps the simple loop $\eta_j$ around $z_j$ to the element $\gamma_j$. Moreover, after an appropriate choice of the elements $\gamma_1, \dots, \gamma_4$, we have 
\begin{equation*}
\gamma_1 \cdots \gamma_4=1, 
\end{equation*}
where 1 is the identity element of $\Gamma$.

The next result shows that the group homomorphism $\tilde{\rho}$ can be used to define a Riemann surface $\tilde{M}$ on which $\Gamma$ acts faithfully.

\begin{proposition}
\label{tildeMprop}
Given the surface $M$ and the group homomorphism $\tilde{\rho}: \pi_1(\CPone \smallsetminus \{z_1, \dots, z_4\}, z_0)\to \Gamma$ defined in \eqref{tildemonodromy}, there exists a Riemann surface $\tilde{M}$ such that:
\begin{itemize}
\item[$(1)$]
The group $\Gamma$ acts faithfully on $\tilde{M}$ and the quotient $\tilde{M}/\Gamma$ is the Riemann sphere $\CPone$;
\item[$(2)$]
The holomorphic map $\tilde{\pi}: \tilde{M} \to \tilde{M}/\Gamma$ is branched at the points $z_1, \dots, z_4\in \CPone$ and its monodromy representation is given by $\tilde{\rho}$;
\item[$(3)$]
There exists a holomorphic covering map $\tau: \tilde{M} \to M$ of degree 2, branched at the fixed points of the action of $G$ on $M$.
\end{itemize}
\end{proposition}

\begin{proof}
The existence of the Riemann surface $\tilde{M}$ and the conditions $(1)$ and $(2)$ come from the Riemann's existence Theorem and its proof (see, for example, \cite[Chapter 3 pp. 90-91]{Mir:tesi}).

Condition $(3)$ comes from the definition of the surface $M$ via the Riemann's existence Theorem using the group homomorphism $\rho$ in \eqref{Mmonodromy} and from the fact that the group $\Gamma\subset \SU(2) \times \SU(2)$ double covers the group $G\subset\SO(4)$.

\end{proof}

We now compute the local eigenvalues of the action of the group $\Gamma$ on the surface $M$.

Let $\tau: \tilde{M} \to M$ be the double covering defined in Proposition \ref{tildeMprop}. The map $\tau$ is branched at the points $P_r$ and $Q_s$ for $r=1, \dots, 2l+2$ and $s = 1, \dots, 2k+2$.

Assume that $k$ is an odd number and $l$ is an even number. The preimages $\tilde{\gamma}_1, \gamma_1 \in \Z_{2k}\times \Z_{2l}$ of the generators $g_1$ for the stabilizer group of the point $P_1$ (the situation for the points $P_{2r+1}$, $r = 1, \dots, l$ is analogous) satisfy
\begin{equation*}
\tilde{\gamma}_1^{2k} = \gamma_1^k = 1.
\end{equation*}

There is a 1:1 correspondence between the stabilizer group of $P_1$ and a subgroup of the stabilizer group $\Z_{2k} \subset \Z_{2k}\times \Z_{2l}$ of the point $\tilde{P}_1 := \tau^{-1}(P_1) \in \tilde{M}$.

Let $z$ be a local coordinate around $P_1$ and $w$ a local coordinate around $\tilde{P}_1$ such that $w^2 = z$. The element $g_1$ acts around $P_1$ as
\begin{equation*}
g_1(z) = e^{\frac{2\pi i}{k}}z
\end{equation*}
and the element $\tilde{\gamma}_1$ acts around $\tilde{P}_1$ as
\begin{equation*}
\tilde{\gamma}_1(w) = e^{\frac{2\pi i}{2k}}w.
\end{equation*}

Up to rotations and sign, we can consider the representation of $\Z_{2k}\times \Z_{2l}$ in $\SU(2) \times \SU(2)$ such that (on the first factor)
\begin{equation*}
\tilde{\gamma}_1 \mapsto
\begin{pmatrix}
e^{\frac{2\pi i}{2k}} & 0\\
0 & e^{-\frac{2\pi i}{2k}}
\end{pmatrix}.
\end{equation*}
Therefore, the element $\gamma_1$ is mapped to the element whose first factor is given by
\begin{equation*}
\gamma_1 \mapsto
-
\begin{pmatrix}
e^{\frac{2\pi i}{2k}} & 0\\
0 & e^{-\frac{2\pi i}{2k}}
\end{pmatrix}
=
\begin{pmatrix}
e^{\frac{2\pi i(k+1)}{2k}} & 0\\
0 & e^{\frac{2\pi i(k-1)}{2k}}
\end{pmatrix}.
\end{equation*}

Since $g_2 = g_1^{-1}$, the preimages $\tilde{\gamma}_2, \gamma_2 \in \Z_{2k}\times \Z_{2l}$ of $g_2$ are such that $\gamma_2$ is mapped, under the representation of $\Z_{2k}\times \Z_{2l}$ in $\SU(2) \times \SU(2)$, to the element whose first factor is given by
\begin{equation*}
\gamma_2 \mapsto
\begin{pmatrix}
e^{\frac{2\pi i(k-1)}{2k}} & 0\\
0 & e^{\frac{2\pi i(k+1)}{2k}}
\end{pmatrix}.
\end{equation*}

The local eigenvalues of the monodromy representation $\tilde{\rho}: \pi_1(\CPone \smallsetminus \{z_1, \dots, z_4\}, z_0) \to \Z_{2k}\times \Z_{2l}$ mapping the simple loops $\eta_1, \eta_2$ around $z_1, z_2$ respectively to $\gamma_1, \gamma_2$, are given by
\begin{equation}
\label{loceigenk}
\frac{k+1}{2k}, \quad \text{ and } \quad \frac{k-1}{2k}.
\end{equation}

Consider now the generator $g_3$ of the stabilizer group of the point $Q_1$ (the situation for the points $Q_{2s+1}$, $s = 1, \dots, k$ is analogous). Its preimages $\tilde{\gamma}_3, \gamma_3 \in \Z_{2k}\times \Z_{2l}$ satisfy
\begin{equation*}
\tilde{\gamma}_3^{2l} = \gamma_3^{2l} = 1.
\end{equation*}

Let $z$ be a local coordinate around $Q_1$ and $w$ a local coordinate around $\tilde{Q}_1 := \tau^{-1}(Q_1) \in \tilde{M}$, such that $w^2 = z$. The element $g_3$ acts, locally around $Q_1$, as
\begin{equation*}
g_3(z) = e^{\frac{2\pi i}{l}}z.
\end{equation*}

Then, the element $\tilde{\gamma}_3$ acts around $\tilde{Q}_1$ as 
\begin{equation*}
\tilde{\gamma}_3(w) = e^{\frac{2\pi i}{2l}}w.
\end{equation*}

Moreover, the element $\tilde{\gamma}_3$ acts as a rotation of the same angle, but in the opposite direction, around $\tilde{Q}_2 := \tau^{-1}(Q_2)\in \tilde{M}$ (and around the points $\tilde{Q}_{2s+2} := \tau^{-1}(Q_{2s+2})$ for $s=1, \dots, k$ in the same way ):
\begin{equation*}
\tilde{\gamma}_3(\tilde{w}) = e^{\frac{2\pi i(2l-1)}{2l}}\tilde{w},
\end{equation*}
where $\tilde{w}$ is a local coordinate around $\tilde{Q}_2$.

Up to rotations and sign, we can consider the representation of $\Z_{2k}\times\Z_{2l}$ in $\SU(2) \times \SU(2)$ such that (on the first factor)
\begin{equation*}
\tilde{\gamma}_3 \mapsto
\begin{pmatrix}
e^{\frac{2\pi il}{2l}} & 0\\
0 & e^{-\frac{2\pi il}{2l}}
\end{pmatrix}.
\end{equation*}

Under the same representation, we obtain
\begin{equation*}
\gamma_3 \mapsto
\begin{pmatrix}
e^{\frac{2\pi i(l+1)}{2l}} & 0\\
0 & e^{\frac{2\pi i(l-1)}{2l}}
\end{pmatrix}
\end{equation*}
and the element $\gamma_4 \in \varepsilon ^{-1}(g_4)$ is mapped to the same element as $\gamma_3^{-1}$.

The local eigenvalues of the monodromy representation $\tilde{\rho}$ mapping the simple loops $\eta_3, \eta_4$ around the points $z_3, z_4 \in \CPone$ to the elements $\gamma_3, \gamma_4 \in \Z_{2k}\times\Z_{2l}$, respectively, are given by
\begin{equation}
\label{loceigenl}
\frac{l+1}{2l}, \quad \text{ and } \quad \frac{l-1}{2l}.
\end{equation}

The computations for the case of $k$ and $l$ being both odd (or both even) numbers can be carry out analogously and the local eigenvalues of the monodromy representation $\tilde{\rho}$ are given by \eqref{loceigenk} and \eqref{loceigenl}.

We will now show how the action of $\Gamma$ on $M$ can be lifted to an action of $\Gamma$ on the vector bundle $E \to M$ where the associated family of flat $\SL(2, \mathbb{C})$-connections of the immersion $f: M \hookrightarrow \Sp^3$ is defined.

For every element $\gamma \in \Gamma$ there is a biholomorphic map, which we denote with the same symbol, $\gamma: M \to M$, where
\begin{equation*}
\gamma(p) := \gamma \cdot p, \quad p \in M.
\end{equation*}

It is possible to lift the action of $\Gamma$ on $M$ to an action of $\Gamma$ on the complex vector bundle $E \to M$ as follows: for each $\gamma \in \Gamma$ we consider the pullback bundle $\gamma^*E \to M$ of $E$, which is isomorphic to $E$. Thus, there is a representation of the group $\Gamma$ into the gauge group $\mathcal{G}$ of the complex vector bundle $E \to M$, given by
\begin{equation}
\label{Gammagaugerep}
\begin{split}
\Gamma &\to \mathcal{G}\\
\gamma &\mapsto g_\gamma,
\end{split}
\end{equation}
where $g_\gamma: E \to \gamma^*E \simeq E$.

The complex vector bundle $E\to M$, together with the action $\Gamma \times E \to E$, is called an \textit{orbifold bundle} (we refer to \cite{Ste:tesi} and \cite{Schaf:tesi} for more details about orbifold bundles).

The action of $\Gamma$ on $E \to M$ induces an action of $\Gamma$ on the space of sections of $E$. In fact, let $s$ be a section of $E$, then
\begin{equation}
\label{actiononsections}
(\gamma \cdot s)(p) := g_{\gamma^{-1}}(p)s(p), \quad p \in M.
\end{equation}

From the fact that \eqref{Gammagaugerep} is a representation, the following hold:
\begin{equation*}
\begin{split}
g_1 &= \Id\\
g_{\gamma_1, \gamma_2} &= g_{\gamma_1} g_{\gamma_2}, \quad \forall \gamma_1, \gamma_2 \in \Gamma.
\end{split}
\end{equation*}
Therefore,  \eqref{actiononsections} is a well-defined $\Gamma$-action on the space of sections of $E$.

We consider the rank 2 vector bundle $\tilde{E} := \tau^*E\to \tilde{M}$ equipped with the action of $\Gamma$ induced by the action of $\Gamma$ on the vector bundle $E\to M$.
The following proposition shows that $\Gamma$ acts also on the associated family of flat connections $\nabla^\lambda$. Moreover, the connections $\nabla^\lambda$ and $\tau^*\nabla^\lambda$ are $\Gamma$-equivariant for every $\lambda \in \C^*$.
\begin{proposition}
\label{equivariantfamilytildeE}
The associated family of flat $\SL(2, \mathbb{C})$-connections $\nabla^\lambda$ of the immersion $f: M \hookrightarrow \Sp^3$ given by \eqref{assfamconnect} is $\Gamma$-equivariant.
Moreover, the pullback connections $\tau^*\nabla^\lambda$ defined over the vector bundle $\tilde{E}\to \tilde{M}$ are $\Gamma$-equivariant, where $\tau: \tilde{M} \to M$ is the holomorphic map of degree 2 defined in Proposition \ref{tildeMprop}.
\end{proposition}
\begin{proof}
Recall that $\nabla^\lambda = \nabla + \lambda^{-1}\Phi - \lambda\Phi^*$, where $\nabla  = d + \frac{1}{2}f^{-1}df$ and $\Phi, \Phi^*$ are, respectively, the $(1, 0)$ and $(0,1)$-components of the 1-form $\omega = f^{-1}df$.

Let $p \in M$ be a point with non trivial stabilizer group with respect to the action of $\Gamma$ on $M$. It is possible to normalize the immersion $f: M \hookrightarrow \Sp^3$ such that
\begin{equation*}
 f(p) = \Id \in \Sp^3 \simeq \SU(2).
\end{equation*}

Consider the action of $\gamma \in \Gamma$ on a neighbourhood of $p \in M$, then we have
\begin{equation*}
\Id = f(p) = (\gamma \circ f)(p) = af(p)b^{-1} = ab^{-1} \Rightarrow a = b, \quad a, b\in \SU(2).
\end{equation*}

Therefore, locally around a point $p$ with non trivial stabilizer group, $\Gamma$ acts on $M$ by the conjugation of an element $b \in \SU(2)$.

Let $\omega = f^{-1}df$ be the connection 1-form of the connection $\nabla$ on $E \to M$. Around the point $p \in M$, $\Gamma$ acts on $\omega$ as
\begin{equation}
\label{Gammaactionforms}
\begin{split}
\gamma^*\omega &= (\gamma \circ f)^{-1}d(\gamma \circ f)\\
&=bf^{-1}b^{-1}bdfb^{-1}\\
&= bf^{-1}dfb^{-1} \\
&=b\omega b^{-1}.
\end{split}
\end{equation}

Thus, the local action of $\Gamma$ on the connection $\nabla$ is given by the constant gauge
\begin{equation*}
\gamma \cdot \nabla = \nabla \cdot b = b \nabla b^{-1}.
\end{equation*}

From the fact that $\Phi$ and $\Phi^*$ are, respectively, the $(1,0)$ and $(0,1)$ components of the 1-form $\omega$ and from \eqref{Gammaactionforms} we conclude that $\Gamma$ acts, locally, on the associated family of flat $\SL(2, \mathbb{C})$-connections $\nabla^\lambda = \nabla + \lambda^{-1}\Phi - \lambda \Phi^*$ of the immersion $f: M \to \Sp^3$ as
\begin{equation}
\label{famflatconnectionsequivariance}
\gamma \cdot \nabla^\lambda = \nabla^\lambda \cdot b.
\end{equation}

The only thing left to prove is the second part of the statement. From the construction of the holomorphic map $\tau$ and the Riemann surface $\tilde{M}$, we have the following commutative diagram
\begin{equation}
\label{tauaction}
\begin{tikzcd}
&\Gamma \times \tilde{M} \arrow[d, "{(\Id, \tau )}"] \arrow[r, "\gamma"] &\tilde{M} \arrow[d, "\tau"]\\
&\Gamma \times M \arrow[r, "\gamma"] &M
\end{tikzcd}
\end{equation}

Given $\gamma \in \Gamma$, from the commutativity of the diagram \eqref{tauaction}, the equation \eqref{famflatconnectionsequivariance} and the properties of the pullback we obtain
\begin{equation*}
\begin{split}
\gamma^*( \tau^*\nabla^\lambda) &= (\tau \circ \gamma)^*\nabla^\lambda = (\gamma \circ \tau)^*\nabla^\lambda\\
& = \tau^*(\gamma^*\nabla^\lambda)= \tau^*(\nabla^\lambda \cdot b) \\
&= \tau^*(b^{-1}\nabla^\lambda b) = b^{-1}\tau^*\nabla^\lambda b\\
& = (\tau^*\nabla^\lambda) \cdot b.
\end{split}
\end{equation*}
\end{proof}

\section{Parabolic bundles and logarithmic connections}
\label{SecParbund}
Following \cite{Bis:tesi} we want to define a \textit{parabolic bundle} over $\CPone$ which makes possible the construction of a $\lambda$-family of Fuchsian systems on $\CPone$ for $\lambda \in \mathbb{C}$.

A \textit{parabolic structure} on a holomorphic bundle $E \to M$ over a compact Riemann surface $M$ is given by a filtration of sub-bundles 
\begin{equation}
\label{parfiltration}
E = F_1(E) \supset F_2(E) \supset \dots \supset F_l(E) \supset F_{l+1}(E) = E(-D),
\end{equation}
where $D$ is an effective divisor on $M$, together with a system of \textit{parabolic weights} $\{\alpha_1, \dots, \alpha_l\}$, such that 
\begin{equation*}
0 \leq \alpha_1 < \dots < \alpha_l < 1.
\end{equation*}
The weight $\alpha_j$ corresponds to the sub-bundle $F_j(E) \subset E$.
A holomorphic vector bundle $E$ equipped with a parabolic structure is called a \textit{parabolic bundle}.

Given a parabolic bundle $E \to M$, the sub-bundles $E_t := F_j(E)(- \lfloor t\rfloor D), t \in \R$ give a decreasing, left continuous filtration of sub-bundles of $E$. We say that the filtration $\{E_t\}$ has a \textit{jump} in $t\in \R$ if, for any $\varepsilon >0$ we have $E_{t+\varepsilon} \neq E_t$.

For $M$ being the Lawson surface of genus $(k-1)(l-1)$, $\Gamma = \Z_{2k}\times\Z_{2l}$ has order $4kl$. Therefore, the push-forward bundle $\tilde{\pi}_*\tilde{E} \to \CPone$ of $\tilde{E}\to \tilde{M}$ under the map $\tilde{\pi}$ is a rank $8kl$ bundle over the Riemann sphere.

The action of $\Gamma$ on $\tilde{E}\to \tilde{M}$ induces an action of $\Gamma$ on $\tilde{\pi}_*\tilde{E} \to \CPone$ and it is possible to consider the $\Gamma$-invariant sections of $\tilde{\pi}_*\tilde{E}$. Biswas \cite{Bis:tesi} showed that these sections generate a rank 2 $\Gamma$-invariant vector bundle which we will denote with $(\tilde{\pi}_*\tilde{E})^\Gamma \to \CPone$. This will be the rank 2 vector bundle over $\CPone$ where the family of Fuchsian systems will be defined.

First, we want to define a parabolic structure on $(\tilde{\pi}_*\tilde{E})^\Gamma \to \CPone$. Following \cite{Bis:tesi}, we have that the sub-bundles
\begin{equation}
\label{BisFiltr}
\tilde{E}_t := \Bigg( \tilde{\pi}_*\bigg( \tilde{E} \otimes L \Big( \sum_{j=1}^4 \lfloor -2t  d_j \rfloor\tilde{\pi}^{-1}(z_j) \Big) \bigg) \Bigg)^\Gamma,\quad t \in \mathbb{R},
\end{equation}
give a parabolic filtration on $(\tilde{\pi}_*\tilde{E})^\Gamma \to \CPone$, where $z_1, \dots, z_4$ are the branch points of the map $\tilde{\pi}: \tilde{M} \to \CPone$ and $2d_j$ is the order of the stabilizer group $\Gamma_{p_j^k}$ of the points $p_j^k \in \tilde{\pi}^{-1}(z_j) \subset \tilde{M}, j = 1, \dots, 4$.
Moreover, if $\alpha_0, \dots, \alpha_l$ are the values of $t \in \R$ such that $\{\tilde{E}_t\}$ has a jump, the filtration \eqref{BisFiltr} together with the weights $\alpha_0, \dots, \alpha_l$ defines a parabolic structure on $(\tilde{\pi}_*\tilde{E})^\Gamma \to \CPone$.

Next, we want to construct another parabolic structure, equivalent to the one obtained by Biswas, on $(\tilde{\pi}_*\tilde{E})^\Gamma \to \CPone$ using a logarithmic connection.


Following \cite{BisHe:tesi}, we prove the next result, which provides the existence of a logarithmic connection on $(\tilde{\pi}_*\tilde{E})^\Gamma$.

\begin{proposition}
\label{equivlogconnection}
Let $M$ be the Lawson surface of genus $(k-1)(l-1)$ and $\tilde{\pi}: \tilde{M} \to \CPone$ the holomorphic map defined in Proposition \ref{tildeMprop}. There exists a logarithmic connection $\tilde{\nabla}$ on the $\Gamma$-invariant vector bundle $(\tilde{\pi}_*\tilde{E})^\Gamma \to \CPone$.
\end{proposition}
\begin{proof}

From \cite[Lemma 2.1]{BisHe:tesi} we have an inclusion map
\begin{equation}
\label{inclusionmap}
H^0(U, K_{\tilde{M}}) \hookrightarrow H^0(U, \tilde{\pi}^*(K_{\CPone} \otimes L(D))),
\end{equation}
where $D = z_1 + \dots + z_4$ is the branch divisor of the map $\tilde{\pi}: \tilde{M} \to \CPone$.

Let $U \subset \tilde{M}$ be an open set and $\tau^*\nabla: H^0(U, \tilde{E})\to H^0(U, \tilde{E} \otimes K_{\tilde{M}})$ be the $\Gamma$-equivariant connection on the holomorphic vector bundle $\tilde{E} \to \tilde{M}$.

The composition of $\tau^*\nabla$ with the inclusion map \eqref{inclusionmap} gives a map
\begin{equation*}
h: H^0(U, \tilde{E})\to H^0(U, \tilde{E} \otimes \tilde{\pi}^*(K_{\CPone} \otimes L(D))).
\end{equation*}

The map $h$ induces a map on the push-forward bundles
\begin{equation*}
\tilde{\pi}_*(\tau^*\nabla): H^0(\tilde{U}, \tilde{\pi}_*\tilde{E})\to H^0(\tilde{U}, \tilde{\pi}_*(\tilde{E}\otimes \tilde{\pi}^*(K_{\CPone} \otimes L(D)))),
\end{equation*}
where $\tilde{U} \subset \CPone$ is an open set such that $\tilde{\pi}(U) = \tilde{U}$.

Moreover, from \cite[Proposition 4.2]{Hit2:tesi} we obtain
\begin{equation*}
H^0(\tilde{U}, \tilde{\pi}_*(\tilde{E}\otimes \tilde{\pi}^*(K_{\CPone} \otimes L(D)))) = H^0(\tilde{U}, \tilde{\pi}_*\tilde{E} \otimes K_{\CPone} \otimes L(D)).
\end{equation*}

We can conclude that the map 
\begin{equation*}
\tilde{\pi}_*(\tau^*\nabla): H^0(\tilde{U}, \tilde{\pi}_*\tilde{E})\to H^0(\tilde{U}, \tilde{\pi}_*\tilde{E} \otimes K_{\CPone} \otimes L(D))
\end{equation*} 
satisfies the Leibniz rule, since $\tau^*\nabla$ satisfies it and gives a logarithmic connection on the holomorphic vector bundle $\tilde{\pi}_*\tilde{E}$.

Finally, the logarithmic connection $\tilde{\pi}_*(\tau^*\nabla)$ induces a logarithmic connection on the $\Gamma$-invariant bundle $(\tilde{\pi}_*\tilde{E})^\Gamma$. In fact, let $s \in H^0(\tilde{U}, \tilde{\pi}_*\tilde{E})$ be a $\Gamma$-invariant section and $\gamma \in \Gamma$, then 
\begin{equation*}
\begin{split}
\tilde{\pi}_*(\tau^*\nabla) (s) &= \tilde{\pi}_*(\tau^*\nabla)(\gamma \cdot s) = \tau^*\nabla(\gamma \cdot s)\\
&= \gamma^*(\tau^*\nabla)(s) = \gamma^*(\tilde{\pi}_*(\tau^*\nabla))(s),
\end{split}
\end{equation*}
where we have used the fact that $s$ can be considered as a section of $\tilde{E}$ and that $\tau^*\nabla$ is $\Gamma$-equivariant.

We denote with $\tilde{\nabla}$ the logarithmic connection on $(\tilde{\pi}_*\tilde{E})^\Gamma$ induced by the connection $\tilde{\pi}_*(\tau^*\nabla)$.

\end{proof}

In order to compute the local residues of the logarithmic connection $\tilde{\nabla}$ at the branch points $z_1, \dots, z_4 \in \CPone$ of the map $\tilde{\pi}$, consider $p \in \tilde{\pi}^{-1}(z_j) \subset \tilde{M}$. Let $2d_j-1$ be the branch order of $p$ \footnote{We recall that $d_j$ is equal to $k$ for $j=1, 2$ and it is equal to $l$ for $j=3, 4$}. Then, it is possible to find a local coordinate $w$ around $p$ and a local coordinate $z$ around $z_j$ such that
$$
w^{2d_j} = z.
$$

Up to shrinking the domain of the local coordinate $w$, we can consider it defined on a trivializing set $U \subset \tilde{M}$ for the holomorphic vector bundle $\tilde{E} \to \tilde{M}$. If $(s, t)$ is a holomorphic local frame for $\tilde{E} \to U$,  and the stabilizer group $\Gamma_p$ of the point $p \in \tilde{M}$ under the action of $\Gamma$ acts on $U$ as the cyclic group $\Z_{2d_j}$, a simple computation shows that a local frame around $z_j \in \CPone$ for $(\tilde{\pi}_*\tilde{E})^\Gamma$ is given by $(w^{d_j-1}s, w^{d_j+1}t)$.

Therefore, locally we can write the logarithmic connection $\tilde{\nabla}$ as
\begin{equation*}
\tilde{\nabla} = d + 
\begin{pmatrix}
(d_j-1)\frac{dw}{w} & 0\\
0 & (d_j+1)\frac{dw}{w}
\end{pmatrix}
\end{equation*}
and we conclude that, with respect to the frame $(w^{d_j-1}s, w^{d_j+1}t)$, the local residue of $\tilde{\nabla}$ at $z_j$ is given by
\begin{equation}
\label{locresiduetildenabla}
\Res_{z_j}\tilde{\nabla} = 
\begin{pmatrix}
\frac{(d_j-1)}{2d_j} & 0\\
0 & \frac{(d_j+1)}{2d_j}
\end{pmatrix}.
\end{equation}

Moreover, the Riemann-Hilbert (\cite[Theorem 3.6]{Heu:tesi}) correspondence and the construction of $\tilde{\nabla}$ imply the following
\begin{proposition}
\label{gaugeequivconnections}
The connection $\tau^*\nabla$ on the holomorphic vector bundle $\tilde{E}\to \tilde{M}$ is gauge equivalent to the pullback connection $\tilde{\pi}^*\tilde{\nabla}$ of the logarithmic connection defined on the $\Gamma$-invariant vector bundle $(\tilde{\pi}_*\tilde{E})^\Gamma\to \CPone$, under the holomorphic map $\tilde{\pi}: \tilde{M} \to \CPone$.
\end{proposition}

Having a logarithmic connection on $(\tilde{\pi}_*\tilde{E})^\Gamma\to \CPone$ gives us the possibility to define a parabolic structure on the same holomorphic vector bundle as follows:

Let $\Res_{z_j}\tilde{\nabla} \in \End((\tilde{\pi}_*\tilde{E})^\Gamma)_{z_j}$ be the local residue of $\tilde{\nabla}$ at the branch point $z_j \in \CPone$ and $\mu_1^j, \mu_2^j$ its eigenvalues, with $\mu_2^j \geq \mu_1^j$.

The eigenline $L_j := \Ker(\Res_{z_j}\tilde{\nabla} - \mu_2^j \Id)$ of $\Res_{z_j}\tilde{\nabla}$ corresponding to $\mu_2^j$ is contained in the fiber of $(\tilde{\pi}_*\tilde{E})^\Gamma$ at $z_j$ and we obtain a filtration
\begin{equation}
\label{parfiltrlogconnection}
0 \subset L_j \subset ((\tilde{\pi}_*\tilde{E})^\Gamma)_{z_j}
\end{equation}
We will call the eigenline $L_j$ the \textit{parabolic line} of $\tilde{\nabla}$ at the branch point $z_j \in \CPone$.

The filtration \eqref{parfiltrlogconnection} together with the system of weights $\{\mu_1^j, \mu_2^j\}$ gives a parabolic structure on $(\tilde{\pi}_*\tilde{E})^\Gamma$.

We want to show that the parabolic structure on $(\tilde{\pi}_*\tilde{E})^\Gamma$ defined via \eqref{parfiltrlogconnection} is equivalent to the parabolic structure defined by Biswas in \cite{Bis:tesi}, using the parabolic filtration \eqref{BisFiltr}.

%

The parabolic filtration for the vector bundle $(\tilde{\pi}_*\tilde{E})^\Gamma \to \CPone$ defined by Biswas is given by (cf. \eqref{BisFiltr})
\begin{equation*}
(\tilde{\pi}_*\tilde{E})^\Gamma_t = \Bigg( \tilde{\pi}_*\bigg( \tilde{E} \otimes L\Big( \sum_{j=0}^l \lfloor -2kt \rfloor P_{2j+1} + \sum_{j=0}^l \lfloor -2kt \rfloor P_{2j+2} + \sum_{j=0}^k \lfloor -2lt \rfloor Q_{2j+1} + \sum_{j=0}^k \lfloor -2lt \rfloor Q_{2j+2} \Big) \bigg) \Bigg)^\Gamma,
\end{equation*}
for $t \in \mathbb{R}$.

We study in detail the situation around the points $P_{2r+1}$. For the points $P_{2r+2}$ the computations are the same and for the points $Q_s$ it is only sufficient to use the integer $l$ instead of $k$.

There are three cases, according to the value of $t \in \mathbb{R}$, to consider
\begin{itemize}
\item
$-\frac{k-1}{2k} < t \leq \frac{k-1}{2k}$

\begin{equation*}
(\tilde{\pi}_*\tilde{E})^\Gamma_{t_1} = \Bigg( \tilde{\pi}_*\bigg( \tilde{E} \otimes L\Big( \sum_{r=0}^l -(k-1) P_{2r+1} \Big) \bigg) \Bigg)^\Gamma.
\end{equation*}

The bundle $(\tilde{\pi}_*\tilde{E})^\Gamma_{t_1}$ is generated by the sections of $\tilde{\pi}_*\tilde{E}$ which are $\Gamma_{P_{2r+1}}$-invariant and have at least a $(k-1)$-order zero at the points $P_{2r+1}$ when considered as sections of the holomorphic vector bundle $\tilde{E}$.

We can consider the  local frame, around the point $z_1 \in \CPone$, of $(\tilde{\pi}_*\tilde{E})^\Gamma_{t_1}$ given by
\begin{equation}
\label{locframeEt1}
(w^{k-1}s, w^{k+1}t),
\end{equation}
where $(s, t)$ is a holomorphic local frame of $\tilde{E}\to \tilde{M}$ around $P_{2r+1}$. 

A local frame for $(\tilde{\pi}_*\tilde{E})^\Gamma$ around $z_1 \in \CPone$ is given by $(w^{k-1}s, w^{k+1}t)$. Therefore, the bundles $(\tilde{\pi}_*\tilde{E})^\Gamma$ and $(\tilde{\pi}_*\tilde{E})^\Gamma_{t_1}$ coincides.

\item
$\frac{k-1}{2k} < t \leq \frac{k+1}{2k}$

\begin{equation*}
(\tilde{\pi}_*\tilde{E})^\Gamma_{t_2} = \Bigg( \tilde{\pi}_*\bigg( \tilde{E} \otimes L\Big( \sum_{r=0}^l -(k+1) P_{2r+1} \Big) \bigg) \Bigg)^\Gamma.
\end{equation*}

Similarly to the previous case, the bundle $(\tilde{\pi}_*\tilde{E})^\Gamma_{t_2}$ is generated by the $\Gamma_{P_{2r+1}}$-invariant sections of $\tilde{\pi}_*\tilde{E}$ having at least a $(k+1)$-order zero at $P_{2r+1}$ when considered as sections of the bundle $\tilde{E}$.

We can consider the  local frame for $(\tilde{\pi}_*\tilde{E})^\Gamma_{t_2}$ given by
\begin{equation}
\label{locframeEt2}
(w^{2k}s, w^{k+1}t),
\end{equation}
where $(s, t)$ is a holomorphic local frame for $\tilde{E}$ around $P_{2r+1}$.

Comparing the local frame \eqref{locframeEt1} for $(\tilde{\pi}_*\tilde{E})^\Gamma_{t_1}$ with the local frame \eqref{locframeEt2} for $(\tilde{\pi}_*\tilde{E})^\Gamma_{t_2}$, we observe that we have an inclusion of vector bundles
\begin{equation*}
(\tilde{\pi}_*\tilde{E})^\Gamma_{t_2} \subset (\tilde{\pi}_*\tilde{E})^\Gamma_{t_1}.
\end{equation*}

Moreover, the parabolic line at the point $z_1\in \CPone$ (the fiber of $(\tilde{\pi}_*\tilde{E})^\Gamma_{t_2}$) is given, with respect to these local frames, by
\begin{equation*}
[0: \tilde{t}(z_1)],
\end{equation*}
where $\tilde{t} := w^{k+1}t$.

We recall that the local residue of the logarithmic connection $\tilde{\nabla}$ on $(\tilde{\pi}_*\tilde{E})^\Gamma$ at the point $z_1$ can be written as 
\begin{equation*}
\Res_{z_1}\tilde{\nabla} = 
\begin{pmatrix}
\frac{(k-1)}{2k} & 0\\
0 & \frac{(k+1)}{2k}
\end{pmatrix}
\end{equation*}
with respect to the local frame $(\tilde{s}, \tilde{t}) = (w^{k-1}s, w^{k+1}t)$ of $(\tilde{\pi}_*\tilde{E})^\Gamma$ around $z_1 \in \CPone$.

Therefore, the parabolic line $L_1$ given by \eqref{parfiltrlogconnection}, is the line $[0: \tilde{t}(z_1)]$. We can conclude that the parabolic line at $z_1$ for the parabolic structures on $(\tilde{\pi}_*\tilde{E})^\Gamma$ defined via \eqref{BisFiltr} and via the logarithmic connection $\tilde{\nabla}$ are the same.

\item
$\frac{k+1}{2k} < t \leq 1+ \frac{k-1}{2k}$

\begin{equation*}
(\tilde{\pi}_*\tilde{E})^\Gamma_{t_3} = \Bigg( \tilde{\pi}_*\bigg( \tilde{E} \otimes L\Big( \sum_{r=0}^l -2\tilde{k} P_{2r+1} \Big) \bigg) \Bigg)^\Gamma, \quad \tilde{k} \geq k.
\end{equation*}

In this case we have to consider $\Gamma_{P_{2r+1}}$-invariant sections of $(\tilde{\pi}_*\tilde{E})$ which have at least a $2k$-order zero at $P_{2r+1}$ when considered as sections of $\tilde{E}$.
A local frame for $(\tilde{\pi}_*\tilde{E})^\Gamma_{t_3}$ is given by
\begin{equation*}
(w^{2k}s, w^{2k}t), 
\end{equation*}
where $(s, t)$ are as above.
We recall that the local coordinates $w$ and $z$ satisfy $w^{2k} = z$. Therefore, the bundle $(\tilde{\pi}_*\tilde{E})^\Gamma_{t_3}$ is locally generated by the sections $(zs, zt)$, which give a holomorphic local frame for the bundle
\begin{equation*}
(\tilde{\pi}_*\tilde{E})^\Gamma \otimes L(-z_1) = (\tilde{\pi}_*\tilde{E})^\Gamma(-z_1).
\end{equation*}
\end{itemize}

We can conclude that the parabolic filtration, given by $\{(\tilde{\pi}_*\tilde{E})^\Gamma_{t_i}\}$, $i=1, \dots, 3$, at the branch points $z_1, \dots, z_4 \in \CPone$ of the holomorphic map $\tilde{\pi}:\tilde{M} \to \CPone$, is given by
\begin{equation*}
((\tilde{\pi}_*\tilde{E})^\Gamma)_{z_j} \subset L_j \subset ((\tilde{\pi}_*\tilde{E})^\Gamma(-z_j))_{z_j},\quad j=1, \dots, 4,
\end{equation*}
where $L_j$ is the parabolic line defined in \eqref{parfiltrlogconnection}.

By patching together the local frames for the bundles $(\tilde{\pi}_*\tilde{E})^\Gamma_{t_j}, j = 1, 2, 3$ around the branch points $z_1, \dots, z_4 \in \CPone$ via the transition functions, it is possible to show that the global parabolic structure induced by the logarithmic connection $\tilde{\nabla}$ is equivalent to the parabolic structure defined by Biswas in \cite{Bis:tesi}. We reassume this fact in the following

\begin{proposition}
\label{equivalenceparstructures}
The logarithmic connection $\tilde{\nabla}$ on the $\Gamma$-invariant vector bundle $(\tilde{\pi}_*\tilde{E})^\Gamma\to \CPone$ given by Proposition \ref{equivlogconnection}, induces a parabolic structure on $(\tilde{\pi}_*\tilde{E})^\Gamma$ via the parabolic filtration \eqref{parfiltrlogconnection}, with parabolic weights given by the eigenvalues of the local residues of $\tilde{\nabla}$ at the branch points $z_1, \dots, z_4$ of the map $\tilde{\pi}$.

Moreover, this parabolic structure is equivalent to the parabolic structure defined by Biswas \cite{Bis:tesi} using \eqref{BisFiltr}, with jumps at the values of $t$ equal to the eigenvalues of the local residues of the connection $\tilde{\nabla}$ at the points $z_1, \dots, z_4$.
\end{proposition}

From what we have showed so far, it is possible to define a $\lambda$-family of logarithmic connections $\tilde{\nabla}^\lambda$, for $\lambda \in \mathbb{C}^*$, on the $\Gamma$-invariant vector bundle $(\tilde{\pi}_*\tilde{E})^\Gamma$. Moreover, the family $\tilde{\nabla}^\lambda$ induces a $\lambda$-family of parabolic structures on the same vector bundle.

We denote with $(\tilde{\pi}_*\tilde{E})^{\Gamma, \lambda}$ the parabolic vector bundle given by $(\tilde{\pi}_*\tilde{E})^\Gamma$ together with the holomorphic structure and the parabolic structure induced by the logarithmic connection $\tilde{\nabla}^\lambda$. 

Proposition \ref{equivalenceparstructures} ensured that it is possible to extend the $\lambda$-family of parabolic bundles $(\tilde{\pi}_*\tilde{E})^{\Gamma, \lambda}$ at $\lambda=0$. In fact, the construction due to Biswas \cite{Bis:tesi} does not depend on $\lambda$.

The next result shows that, for $\lambda \neq 0$, there are only two possible holomorphic structures on $(\tilde{\pi}_*\tilde{E})^{\Gamma, \lambda}$.

\begin{proposition}
\label{possiblholomstructures}
The holomorphic vector bundle $(\tilde{\pi}_*\tilde{E})^{\Gamma, \lambda}\to \CPone$, for $\lambda\neq 0$, with holomorphic structure given by $\bar{\partial}^\lambda := (\tilde{\nabla}^\lambda)^{0,1}$, can be only one of the following two bundles
\begin{equation*}
(\tilde{\pi}_*\tilde{E})^{\Gamma, \lambda} = 
\begin{cases}
\mathcal{O}(-2) \oplus \mathcal{O}(-2)\\
\mathcal{O}(-1) \oplus \mathcal{O}(-3).
\end{cases}
\end{equation*}
\end{proposition}

\begin{proof}
Let $z_1, \dots, z_4$ be the branch points of the holomorphic map $\tilde{\pi}:\tilde{M} \to \CPone$ (cf. Proposition \ref{tildeMprop}) and $2d_j$, $j=1, \dots, 4$, the order of the stabilizer group of the points $p \in \tilde{\pi}^{-1}(z_j)$ under the action of $\Gamma \subset \SU(2) \times \SU(2)$ on $\tilde{M}$.

Equation \eqref{locresiduetildenabla} shows that the eigenvalues of the local residues of the connection $\tilde{\nabla}^\lambda$ at the points $z_1, \dots, z_4$ do not depends on $\lambda$ and are given by
\begin{equation*}
\mu^j_1 = \frac{d_j-1}{2d_j}, \quad \mu^j_2 = \frac{d_j+1}{2d_j},\quad j=1, \dots, 4.
\end{equation*}

The degree of the holomorphic vector bundle $(\tilde{\pi}_*\tilde{E})^{\Gamma, \lambda}\to \CPone$ can be computed using the formula (\cite[Proposition 1.2]{BisDa:tesi})
\begin{equation}
\label{degparbundle}
\degree((\tilde{\pi}_*\tilde{E})^{\Gamma, \lambda}) = - \sum_{j=1}^4(\mu^j_1 + \mu^j_2) = -4.
\end{equation}

Therefore, the Grothendieck Splitting Theorem \cite{Gro:tesi} implies 
\begin{equation*}
(\tilde{\pi}_*\tilde{E})^{\Gamma, \lambda} = \mathcal{O}(n) \oplus \mathcal{O}(m), \quad \lambda \in \mathbb{C}^*,
\end{equation*}
with $m+n = -4$.

Suppose that $n \in \mathbb{N}$ and $m \leq -4$ (or viceversa). The logarithmic connection $\tilde{\nabla}^\lambda$ induces a logarithmic connection on the line sub-bundle $\mathcal{O}(m)\to \CPone$.

Let $\alpha_1, \dots, \alpha_4$ be the local residues of such logarithmic connection on $\mathcal{O}(m)$, which must satisfies
\begin{equation*}
-4 \geq m = - \sum_{j=1}^4\alpha_j.
\end{equation*}

However, the values $\alpha_1, \dots, \alpha_4$ are induced by the local residues of the connection $\tilde{\nabla}^\lambda$ and must be contained in the interval $(0, 1)$ since the values $\mu^j_1, \mu^j_2$ are contained in the same interval.

Hence, we obtain
\begin{equation*}
- \sum_{j=1}^4 \alpha_j >-4
\end{equation*}
which gives a contradiction.

We conclude that the only two possibilities for the values of $m$ and $n$ are given by
\begin{equation*}
\begin{cases}
n=m=-2\\
n=-1, m=-3.
\end{cases}
\end{equation*}

\end{proof}

\section{The parabolic Higgs field and the proof of Theorem \ref{MainTheor}}
\label{SecMain}
Given the $\lambda$-family of logarithmic connections $\tilde{\nabla}^\lambda$ on $(\tilde{\pi}_*\tilde{E})^\Gamma \to \CPone$, we want now to study its residue at $\lambda=0$ and show that this gives a so-called parabolic Higgs field over $\CPone$.

We first recall that it is possible to associate to every parabolic bundle $E \to M$, defined via a logarithmic connection $\nabla$, a \textit{parabolic degree} defined as
\begin{equation*}
\pardeg(E) := \degree(E) + \sum_j \mu_1^j + \mu_2^j,
\end{equation*}
where $\mu_1^j, \mu_2^j$ are the eigenvalues of the local residues of $\nabla$ at the singular points $p_j \in M$.

Moreover, for every parabolic line sub-bundle $V$ of $E$, the parabolic degree of $V$ is given by
\begin{equation}
\label{parsubpardegree}
\pardeg(V) = \degree(V) + \sum_{j=1}^n\alpha_j,
\end{equation}
where
\begin{equation*}
\alpha_j = 
\begin{cases}
\mu^j_2 \quad \text{ if } V_{p_j} = L_j\\
\mu^j_1 \quad \text{ otherwise }
\end{cases}
\end{equation*}
and $L_j$ is the parabolic line of $E\to M$ at the point $p_j \in M$.

A parabolic bundle $E\to M$ is said \textit{parabolic semi-stable} (resp. \textit{parabolic stable}) if for every parabolic sub-bundle $V$ of $E$ with $0 < \rank(V) < \rank(E)$
\begin{equation*}
\frac{\pardeg(V)}{\rank(V)} \leq \frac{\pardeg(E)}{\rank(E)} \quad \bigg(resp. \quad \frac{\pardeg(V)}{\rank(V)} < \frac{\pardeg(E)}{\rank(E)}   \bigg).
\end{equation*}

In \cite[Lemma 4.1]{Tesi} it is shown that the $\Gamma$-invariant parabolic bundle $(\tilde{\pi}_*\tilde{E})^{\Gamma, \lambda} \to \CPone$, together with the logarithmic connection $\tilde{\nabla}^\lambda$ has parabolic degree equal to zero.

Now, let $M$ be the Lawson surface of genus $(k-1)(l-1)$, $\nabla^\lambda$ the associated family of flat $\SL(2, \mathbb{C})$-connections of the immersion $f: M \hookrightarrow \Sp^3$ and $\Phi \in H^0(M, \End_0(E)\otimes K_M)$ its Higgs field. 

If we consider the pullback of $\Phi$ under the map $\tau: \tilde{M} \to M$ as a map 
\begin{equation*}
\tau^*\Phi: H^0(U, \tilde{E}) \to H^0(U, \tilde{E} \otimes K_{\tilde{M}}), \quad U \subset \tilde{M},
\end{equation*}
it is possible to use analogous arguments as in Proposition \ref{equivlogconnection} to prove the following

\begin{proposition}
\label{parabolicHiggsfield}

Let $\Phi$ be the Higgs field of the immersion $f: M \hookrightarrow \Sp^3$. Then, the map $\tilde{\pi}_*(\tau^*\Phi)$ induces a Higgs field $\tilde{\Phi} \in H^0(\tilde{U}, \End_0((\tilde\pi_*\tilde{E})^\Gamma)\otimes K_{\CPone}\otimes L(D))$ on the $\Gamma$-invariant vector bundle $(\tilde{\pi}_*\tilde{E})^\Gamma \to \CPone$.
\end{proposition}

The Higgs field $\tilde{\Phi}$ on $(\tilde{\pi}_*\tilde{E})^\Gamma \to \CPone$ satisfies the following properties:
\begin{proposition}
\label{parhiggsfieldprop}
The Higgs field $\tilde\Phi \in H^0(\CPone, \End_0((\tilde\pi_*\tilde{E})^\Gamma) \otimes K_{\CPone}\otimes L(D))$ defined in Proposition \ref{parabolicHiggsfield} satisfies the following:
\begin{itemize}
\item[$(i)$]
$\tilde\Phi$ is nilpotent, that is, $\tilde\Phi^2 = 0$;
\item[$(ii)$]
$\tilde\phi_j := \Res_{z_j}(\tilde\Phi)\neq 0$, for $j=1, \dots, 4$;
\item[$(iii)$]
$\tilde\Phi$ is a parabolic Higgs field, that is
\begin{equation*}
L_j \in \Ker(\tilde\phi_j), \quad j=1, \dots, 4,
\end{equation*}
where $L_j$ is the parabolic line of $(\tilde\pi_*\tilde{E})^\Gamma$ at $z_j$;
\item[$(iv)$]
for every holomorphic line sub-bundle $L \subset (\tilde\pi_*\tilde{E})^\Gamma$ with $\tilde\Phi(L) \subset L \otimes K$ the following holds
\begin{equation*}
\pardeg(L) < \pardeg((\tilde\pi_*\tilde{E})^\Gamma) = 0.
\end{equation*}
If this condition is satisfied we say that $((\tilde\pi_*\tilde{E})^\Gamma, \tilde\Phi)$ \textit{is a parabolic Higgs stable bundle}.
\end{itemize}
\end{proposition}
\begin{proof}
\begin{itemize}
\item[$(i)$]
Let $\tilde{U} \subset \CPone$ be an open set and $s \in H^0(\tilde{U}, (\tilde{\pi}_*\tilde{E})^\Gamma)$. From the construction of $\tilde{\Phi}$, we have
\begin{equation*}
\tilde{\Phi}(s) = \tau^*\Phi(s),
\end{equation*}
where, in the right hand side, we consider $s$ as a local section of the holomorphic vector bundle $\tilde{E}$ on an open set $U \subset \tilde{M}$ such that $\tilde{\pi}(U) = \tilde{U}$. Since $\tau^*\Phi$ is nilpotent, it follows that $\tilde{\Phi}$ is nilpotent.
\item[$(ii)$]
Consider the local frame $(w^{d_j-1}s, w^{d_j+1}t)$ of $(\tilde{\pi}_*\tilde{E})^\Gamma$ around one branch point $z_j \in \CPone$ of the map $\tilde{\pi}: \tilde{M} \to \CPone$, where $(s, t)$ is a local frame for $\tilde{E}$ on an open set $U \subset \tilde{M}$.

It is possible to write the Higgs field $\tau^*\Phi$, locally on $U$, as $\tau^*\Phi = A(w)dw$, where
\begin{equation*}
A(w) = 
\begin{pmatrix}
a(w) & b(w)\\
c(w) & -a(w)
\end{pmatrix},
\end{equation*}
for some functions $a, b, c: U \to \mathbb{C}$. Therefore, 
\begin{equation*}
\begin{split}
\tau^*\Phi(s) & = (a(w)s + c(w)t)dw\\
\tau^*\Phi(t) & = (b(w)s -a(w)t)dw.
\end{split}
\end{equation*}

Writing $\tau^*\Phi$ with respect to the local frame $(w^{d_j-1}s, w^{d_j+1}t)$, we obtain
\begin{equation*}
\begin{split}
\tau^*\Phi(w^{d_j-1}s) & = (a(w)w^{d_j-1}s + c(w)w^{d_j-1}t)dw =(a(w)w^{d_j-1}s + \frac{c(w)}{w^2}w^{d_j+1}t)dw \\
\tau^*\Phi(w^{d_j+1}t) & = (b(w)w^{d_j+1}s -a(w)w^{d_j+1}t)dw =(w^2 b(w)w^{d_j-1}s -a(w)w^{d_j+1}t)dw.
\end{split}
\end{equation*}

Hence, the Higgs field $\tilde{\Phi}$, with respect to the local frame $(w^{d_j-1}s, w^{d_j+1}t)$ on the open set $\tilde{U} = \tilde{\pi}(U)$, is given by $\tilde{\Phi} = \tilde{A}(w)dw$, where
\begin{equation*}
\tilde{A}(w) = 
\begin{pmatrix}
a(w) & w^2b(w)\\
\frac{c(w)}{w^2} & -a(w)
\end{pmatrix}.
\end{equation*}

We can conclude that the function $c(w)$ must have a simple zero at $w=0$ and there are no conditions on the functions $a(w)$ and $b(w)$.
Therefore, the residue of $\tilde{\Phi}$ at the point $z_j$ is given by
\begin{equation*}
\begin{pmatrix}
0& 0\\
c_0 & 0
\end{pmatrix}, \quad c_0 \in \mathbb{C}^*,
\end{equation*}
which is non vanishing.
\item[$(iii)$]
Let $\mu^1_1, \mu^1_2$ be the eigenvalues of the local residues of the $\lambda$-family of logarithmic connections $\tilde{\nabla}^\lambda$ on $(\tilde{\pi}_*\tilde{E})^\Gamma \to \CPone$ (cf. \eqref{locresiduetildenabla}).

The parabolic line $L_1$ at the branch point $z_1 \in \CPone$ of the map $\tilde{\pi}:\tilde{M}\to \CPone$, is given by
\begin{equation*}
L_1 = \Ker(\Res_{z_1}\tilde{\nabla}^\lambda + \frac{k+1}{2k}\Id).
\end{equation*}
With respect to the local frame $(w^{k-1}s, w^{k+1}t)$ of $(\tilde{\pi}_*\tilde{E})^\Gamma$ around $z_1$, the parabolic line $L_1$ is given by
\begin{equation*}
L_1 = [0 : w^{k+1}t(z_1)].
\end{equation*}

From $(ii)$, the residue of $\tilde{\Phi}$ at $z_1$ is given, with respect to the same local frame, by
\begin{equation*}
\tilde\phi_1 = 
\begin{pmatrix}
0 & 0\\
c_0 & 0
\end{pmatrix}, \quad c_0 \in \mathbb{C}^*.
\end{equation*}

Therefore, we obtain
\begin{equation*}
\begin{pmatrix}
0 & 0\\
c_0 & 0
\end{pmatrix}
\begin{pmatrix}
0\\
w^{k+1}t(z_1)
\end{pmatrix} = 
\begin{pmatrix}
0\\
0
\end{pmatrix},
\end{equation*}
which implies $L_1 \in \Ker(\tilde\phi_1)$. Analogous computations can be made for the parabolic lines at the other branch points $z_2, z_3, z_4 \in \CPone$.
\item[$(iv)$]
Since the holomorphic vector bundle $\tilde{E} \to \tilde{M}$ is stable (\cite{Hit3:tesi}), this follows from \cite[Theor. 3.1]{NaSt:tesi}. 
\end{itemize}
\end{proof}

Having a parabolic Higgs field $\tilde{\Phi}$ on $\CPone$, allows us to determine which holomorphic structures are admissible for $(\tilde{\pi}_*\tilde{E})^{\Gamma, \lambda}$ at $\lambda=0$.
\begin{proposition}
\label{holstructlambda0}
The holomorphic structure of $(\tilde{\pi}_*\tilde{E})^{\Gamma, \lambda}$, induced by the $\lambda$-family of logarithmic connections $\tilde{\nabla}^\lambda$, at $\lambda=0$ can be either $\mathcal{O}(-2) \oplus \mathcal{O}(-2)$ or $\mathcal{O}(-1) \oplus \mathcal{O}(-3)$.
\end{proposition}
\begin{proof}
Analogously to the proof of Proposition \ref{possiblholomstructures}, we have that $(\tilde{\pi}_*\tilde{E})^{\Gamma, \lambda}$ at $\lambda=0$ is given by $\mathcal{O}(n) \oplus \mathcal{O}(m)$ with $n+m = -4$. Suppose that $n \geq 0$ and $m \leq -4$.

Let $(s, t)$ be a frame for $\mathcal{O}(n) \oplus \mathcal{O}(m)$, where $s$ has divisor $n \cdot 0$ and $t$ divisor $m \cdot 0$, for $0 \in \CPone$. If we write the parabolic Higgs field locally as
\begin{equation*}
\tilde{\Phi}(w) = 
\begin{pmatrix}
a(w)  & b(w)\\
c(w) & -a(w)
\end{pmatrix}dw,
\end{equation*}
for some complex-valued functions $a, b, c$, we have
\begin{equation}
\label{equationPhilocframe}
\tilde{\Phi}(s) = (as+ct)dw.
\end{equation}

Equation \eqref{equationPhilocframe} and the assumption on $m$ and $n$ imply that the function $c$ must be zero and $\tilde{\Phi}$ is upper triangular. From $(i)$ of Proposition \ref{parhiggsfieldprop}, $\tilde{\Phi}$ is nilpotent. Therefore, the function $a$ must be also zero and $\tilde{\Phi}$ is given locally by
\begin{equation}
\label{formhiggsfieldlambda0}
\tilde{\Phi}(w) = 
\begin{pmatrix}
0  & b(w)\\
0 & 0
\end{pmatrix}dw,
\end{equation}

From the fact that the Higgs field $\tau^*\Phi$ on the holomorphic vector bundle $\tilde{E} \to \tilde{M}$ is nowhere vanishing (because the Higgs field $\Phi$ on $M$ is nowhere vanishing), it follows that $\tilde{\Phi}$ must be nowhere vanishing. But the Higgs field $\tilde{\Phi}$ of the form \eqref{formhiggsfieldlambda0} admits points where it vanishes and we obtain a contradiction. 

We conclude that the only possibilities for the values of $m$ and $n$ are
\begin{equation*}
\begin{cases}
n=m=-2\\
n=-1, m=-3
\end{cases}.
\end{equation*}
\end{proof}

Long but straightforward computations (\cite[Subsection 4.5.2, Subsection 4.5.3, Appendix A1, Appendix A2]{Tesi}) show that the only holomorphic structure on $(\tilde{\pi}_*\tilde{E})^{\Gamma, \lambda}$ at $\lambda=0$ which admit a nilpotent parabolic Higgs field $\tilde{\Phi}$ with non vanishing residues such that $((\tilde{\pi}_*\tilde{E})^{\Gamma}, \tilde{\Phi})$ is parabolic Higgs stable is $\mathcal{O}(-2) \oplus \mathcal{O}(-2)$ with the parabolic lines $L_1, \dots, L_4$ at the branch points $z_1, \dots, z_4 \in \CPone$ normalized to correspond to four different points of $\CPone$.

We are now ready to prove that the $\lambda$-family of logarithmic connections $\tilde{\nabla}^\lambda$ defined on the $\Gamma$-invariant parabolic bundle $(\tilde{\pi}_*\tilde{E})^{\Gamma, \lambda} \to \CPone$, gives a DPW potential on the four punctured Riemann sphere, from which it is possible to reconstruct the immersion $f: M \hookrightarrow \Sp^3$. Namely, we have the following
%

\begin{proof}[Proof of Theorem \ref{MainTheor}]
We want to prove that the holomorphic family of logarithmic connections $\tilde{\nabla}^\lambda$, defined in Section \ref{SecParbund} satisfies conditions $(i)-(iv)$.
\begin{itemize}
\item[$(i)$]
Let $\tau: \tilde{M} \to M$ be the holomorphic map between Riemann surfaces defined in Proposition \ref{tildeMprop}, branched at the fixed points of the action of $G \subset \SO(4)$ on $M$.

From the Riemann's existence Theorem, $\tau$ is uniquely determined by a monodromy representation
\begin{equation*}
\rho: \pi_1(M \smallsetminus \{P_1, \dots, P_{2l+2}, Q_1, \dots, Q_{2k+2}\}, P_0) \to \Z_2,
\end{equation*}
where $P_1, \dots, P_{2l+2}, Q_1, \dots, Q_{2k+2}$ are the fixed points of the action $G \times M \to M$ and $P_0 \in M \smallsetminus \{P_1, \dots, P_{2l+2}, Q_1, \dots, Q_{2k+2}\}$.

The representation $\rho$ determines a flat connections $\hat{\nabla}$ on $M$ via the Riemann-Hilbert correspondence (\cite[Theorem 3.5]{Heu:tesi}). Moreover, the construction of $\tau$ and $\tilde{M}$ implies that the pullback connection $\tau^*\hat{\nabla}$ on $\tilde{M}$ is trivial.

From Proposition \ref{gaugeequivconnections}, it follows that the $\lambda$-family of flat connections $\tau^*\nabla^\lambda$ and $\tilde{\pi}^*\tilde{\nabla}^\lambda$ are gauge equivalent. Therefore, if we consider the associated family of flat connections $\nabla^\lambda$ on $M$ and the pullback under the map $\pi: M \to \CPone$ of the $\lambda$-family of logarithmic connections $\tilde{\nabla}^\lambda$ on $\CPone$, an argument similar to the one used in the discussion prior to Proposition \ref{gaugeequivconnections}, shows that there only two possible cases:
\begin{itemize}
\item[$(1)$]
$\nabla^\lambda$ is gauge equivalent to $\tilde{\pi}^*\tilde{\nabla}^\lambda$ under a holomorphic family of meromorphic gauge transformations $g(\lambda)$, which extends to $\lambda=0$;
\item[$(2)$]
$\nabla^\lambda$ is gauge equivalent to $\tilde{\pi}^*\tilde{\nabla}^\lambda \otimes \hat{\nabla}$ under a holomorphic family of meromorphic gauge transformations $g(\lambda)$, which extends to $\lambda=0$.
\end{itemize}

Case $(1)$ cannot occur due to the choice for the eigenvalues of the local residues of the family of logarithmic connections $\tilde{\nabla}^\lambda$ at the points $z_1, \dots, z_4 \in \CPone$ we have made.
Therefore, it remains only the situation given in case $(2)$.

\item[$(ii)$]
Proposition \ref{holstructlambda0} shows that the underlying holomorphic vector bundle of the parabolic bundle $(\tilde{\pi}_*\tilde{E})^{\Gamma, \lambda} \to \CPone$ at $\lambda=0$ is given by $\mathcal{O}(-2)\oplus \mathcal{O}(-2)$.

Since the generic holomorphic rank 2 vector bundle of degree $-4$ on $\CPone$ is $\mathcal{O}(-2)\oplus \mathcal{O}(-2)$, it is possible to find an open neighbourhood $U \subset \mathbb{C}^*$ of $\lambda=0$ such that 
\begin{equation*}
(\tilde{\pi}_*\tilde{E})^{\Gamma, \lambda} = \mathcal{O}(-2)\oplus \mathcal{O}(-2), \quad \text{ for every } \lambda \in U.
\end{equation*}

The expression \eqref{DPWpotentialexpression} for the $\lambda$-family of logarithmic connections $\tilde{\nabla}^\lambda$, for $\lambda \in U$, can be obtained by writing the connection 1-form of $\tilde{\nabla}^\lambda$ with respect to the frame
\begin{equation*}
\bigg( \frac{1}{(z-z_1)(z-z_3)}e_1, \frac{1}{(z-z_1)(z-z_3)}e_2\bigg),
\end{equation*}
where $(e_1, e_2)$ is the meromorphic frame for $\mathcal{O}(-2)\oplus \mathcal{O}(-2)$ such that $e_1, e_2$ have simple poles at the branch points $z_2, z_4 \in \CPone$.

\item[$(iii)$]
Since for a generic $\lambda \in \mathbb{C}^*$ the underlying holomorphic vector bundle of the parabolic bundle $(\tilde{\pi}_*\tilde{E})^{\Gamma, \lambda} \to \CPone$ is $\mathcal{O}(-2) \oplus \mathcal{O}(-2)$, the values of $\lambda \in \mathbb{C}$ such that this not happens form a discrete subset $\tilde{U} \subset \mathbb{C}^*$.

By parametrizing the $\lambda$-family of logarithmic connections $\tilde{\nabla}^\lambda$ of the form \eqref{DPWpotentialexpression} similarly to \cite[Section 2.3]{HeHe:tesi}, it follows that the entries of the $\lambda$-family of 1-forms $\eta(z, \lambda)$ cannot have essential singularities.
Therefore, it is possible to extends holomorphically to $\mathbb{C}^*$ the map
\begin{equation*}
\lambda \mapsto \eta(z, \lambda).
\end{equation*}

The unitarizability of $\tilde{\nabla}^\lambda = d + \eta(z, \lambda)$ for $\lambda \in \Sp^1$ comes from the construction of $\tilde{\nabla}^\lambda$ and from the fact that the connection $\nabla^\lambda$ is unitary for $\lambda \in \Sp^1$.

\item[$(iv)$]
This follows from the description of the local residues of $\tilde{\nabla}^\lambda$ in Section \ref{SecParbund} (cf. \eqref{locresiduetildenabla}).
\end{itemize}

The CMC immersion $f: M \hookrightarrow \Sp^3$ can be constructed via the DPW method using the family of 1-forms $\eta(z, \lambda)$ in \eqref{DPWpotentialexpression}. There are two cases to consider:
\begin{itemize}
\item[$(a)$]
for all $\lambda \in D_1 = \{\lambda \in \mathbb{C}^* \mid |\lambda| \leq 1\}$ the underlying holomorphic vector bundle of $(\tilde{\pi}_*\tilde{E})^{\Gamma, \lambda} \to \CPone$ is $\mathcal{O}(-2)\oplus \mathcal{O}(-2)$.

In order to obtain a closed immersion along non trivial loops in $M$, the monodromy matrices of $\eta(z, \lambda)$ must be simultaneously unitarizable (see for example \cite[Section 7]{SKKR:tesi}).
It is possible to find a unitarizer for the monodromy matrices of $\eta(z, \lambda)$ for all $\lambda \in D_1$ (\cite[Theorem 4]{SKKR:tesi}).

Therefore, the immersion $f: M \hookrightarrow \Sp^3$ constructed using the steps of the DPW method is well-defined.

\item[$(b)$]
There exists a discrete subset $\tilde{U} \subset D_1$ such that the underlying holomorphic vector bundle of $(\tilde{\pi}_*\tilde{E})^{\Gamma, \lambda} \to \CPone$ is $\mathcal{O}(-1)\oplus \mathcal{O}(-3)$.

Similarly to case $(a)$, \cite[Theorem 4]{SKKR:tesi} gives a unitarizer for the monodromy matrices of $\eta(z, \lambda)$, which is singular at the values of $\lambda \in \tilde{U}$.

In order to obtain a well-defined immersion $f: M \hookrightarrow \Sp^3$, in this case it is necessary to apply the Iwasawa factorization (step $(ii)$ of the DPW method) on a disc $D_r$ of radius $r <1$ such that (\cite[Section 3.1]{SKKR:tesi})
\begin{equation}
\label{riwasawacondition}
D_r \cap \tilde{U}  = \emptyset.
\end{equation}
Since $\tilde{U}$ is a discrete subset, there exists a $r <1$ such that \eqref{riwasawacondition} holds. Therefore, the immersion $f: M \hookrightarrow \Sp^3$ constructed via the DPW method is well defined.
\end{itemize}

\end{proof}

\section{Symmetric CMC surfaces}
\label{SecSymsurf}

It is possible to prove a result analogous to Theorem \ref{MainTheor} for compact CMC surfaces into $\Sp^3$ which satisfy the following conditions.

\begin{definition}
\label{symCMCsurf}
Let $M$ be a compact Riemann surface. We say that $M$ is a \textit{symmetric CMC surface} if there exists a CMC embedding $f: M \hookrightarrow \Sp^3$ and the following conditions are satisfied:
\begin{itemize}
\item[$(i)$]
There exists a finite subgroup $G \subset \SO(4)$ with a presentation of the form
\begin{equation*}
G = \langle g_1, \dots, g_4 \mid g_1\cdots g_4 = 1\rangle,
\end{equation*}
which acts faithfully on $f(M) \simeq M$, where 1 denotes the identity element of $G$;
\item[$(ii)$]
The quotient $M/G$ is the Riemann sphere $\CPone$;
\item[$(iii)$]
The projection to the quotient $\pi: M \to \CPone$ is a holomorphic map between Riemann surfaces of degree $|G|$, branched at four points $z_1, \dots, z_4 \in \CPone$\footnote{The group $G$ acts transitively on the set $\pi^{-1}(z)$ for each $z \in \CPone$, that is, for every two points $p, \tilde{p} \in \pi^{-1}(z)$ there exists an element $g \in G$ such that $g\cdot p = \tilde{p}$.
A covering map which satisfies this property is called a \textit{Galois covering} and $G$ the \textit{Galois group} of the covering map (we refer to \cite[Chapter 2]{Sza:tesi} for more details on the theory of Galois coverings)}.
\end{itemize}
\end{definition}

There is no complete classification of symmetric CMC surfaces, however the surfaces in Table \eqref{tabsymsurf} satisfy Definition \ref{symCMCsurf}.

\begin{center}
\begin{equation}
\label{tabsymsurf}
\begin{tabular}{| l | c | c | }
\hline
Surface & Genus & $\SO(4)$-symmetry group\\
\hline
\hline
Lawson's surface $\xi_{(g-1, 1)}$ & $g-1$ & $\Z_{g}\times \Z_2$\\
\hline
Lawson's surface $\xi_{(k-1, l-1)}$ & $(k-1)(l-1)$ & $\Z_k \times \Z_l$\\
\hline
KPS Tetrahedral & 3 & $A_4$\\
\hline
KPS Octahedral & 7 & $S_4$\\
\hline
KPS Cubical & 5 & $S_4$\\
\hline
KPS Icosahedral & 19 & $A_5$\\
\hline
KPS Dodecahedral & 11 & $A_5$\\
\hline
Octahedral join & 11 & $S_4$\\
\hline
Icosahedral join & 29 & $A_5$\\
\hline
\end{tabular}
\end{equation}
\end{center}

Let $M$ be a surface from Table \eqref{tabsymsurf} with symmetry group $G \subset \SO(4)$. Around each point $p \in M$ with non trivial stabilizer group with respect to the action of $G$, the local action around $p$ is given by an action of a cyclic group $\Z_k$, for some $k$, as in the case of the Lawson surface $\xi_{k-1, l-1}$ we considered in the rest of the article. Therefore, it is possible to do the same local computations we have done in Section \ref{SecLift} and prove the following result which extends the existence of a DPW potential for the symmetric CMC surfaces in Table \eqref{tabsymsurf}

\begin{theorem}
Let $M$ be one of the surface in Table \eqref{tabsymsurf} with symmetry group $G \subset \SO(4)$. Let $\nabla^\lambda$ be the associated family of flat $\SL(2, \mathbb{C})$-connections of the immersion $f: M \hookrightarrow \Sp^3$. Then, there exists a holomorphic family of logarithmic connections
\begin{equation*}
\tilde\nabla^\lambda = \lambda^{-1}\tilde\Phi + \tilde\nabla + \text{ higher order terms in $\lambda$ }
\end{equation*}
on the four punctured sphere $\CPone$, singular at the four branch points $z_1, \dots, z_4$ of $\pi : M \to M/G = \CPone$, where $\tilde\Phi$ is a nilpotent $\mathfrak{sl}(2, \mathbb{C})$-valued complex linear 1-form, which satisfies conditions $(i)-(iv)$ of Theorem \ref{MainTheor}. 
Therefore, it is possible to reconstruct the immersion $f: M \hookrightarrow \Sp^3$ from a meromorphic DPW potential on the four punctured sphere.
\end{theorem}

\end{document}